\documentclass[12pt]{amsart}

\usepackage[colorlinks]{hyperref}

\usepackage{fullpage}
\usepackage{stmaryrd}
\usepackage{graphicx}
\usepackage{hyperref}
\usepackage{amsmath,amssymb,amsfonts,mathrsfs,amscd}
\usepackage{newtxtext,newtxmath}
\usepackage[all,cmtip]{xy}

\newtheorem{thm}{Theorem}[section]
\newtheorem{lem}[thm]{Lemma}
\newtheorem{cor}[thm]{Corollary}

\newtheorem{thmA}{Theorem}

\newtheorem{corA}[thmA]{Corollary}

\theoremstyle{definition}

\theoremstyle{remark}

\numberwithin{equation}{section}

\def\HJ#1{\par\medskip\noindent{\bf#1.}\bgroup\it \ }
\def\EHJ{\egroup}

\DeclareMathOperator{\B}{{\rm B}_{\pi}}

\DeclareMathOperator{\I}{{\rm I}_\pi}

\newcommand{\NM}{\vartriangleleft}

\DeclareMathOperator{\Irr}{Irr}
\DeclareMathOperator{\IBr}{IBr}

\DeclareMathOperator{\Syl}{Syl}

\begin{document}

\title{Navarro vertices and lifts in solvable groups}

\author[L. Wang]{Lei Wang}
\address{School of Mathematical Sciences, Shanxi University, Taiyuan, 030006, China.}
\email{wanglei0115@163.com}

\author[P. Jin]{Ping Jin*}
\address{School of Mathematical Sciences, Shanxi University, Taiyuan, 030006, China.}
\email{jinping@sxu.edu.cn}


\thanks{*Corresponding author}


\date{}

\maketitle

\begin{abstract}
Let $Q$ be a $p$-subgroup of a finite $p$-solvable group $G$, where $p$ is a prime,
and suppose that $\delta$ is a linear character of $Q$ with the property that
$\delta(x)=\delta(y)$ whenever $x,y\in Q$ are conjugate in $G$.
In this situation, we show that restriction to $p$-regular elements defines a canonical bijection from the set of those irreducible ordinary characters of $G$ with Navarro vertex $(Q,\delta)$ onto the set of irreducible Brauer characters of $G$ with Green vertex $Q$.
Also, we use this correspondence to examine the behavior of lifts of Brauer characters with respect to normal subgroups.
\end{abstract}

\section{Introduction}
Let $p$ be a fixed prime, and let $G$ be a finite $p$-solvable group.
For each irreducible complex character $\chi\in\Irr(G)$,
Navarro \cite{N2002} introduced a canonical pair $(Q,\delta)$ associated to $\chi$
which is uniquely defined up to $G$-conjugacy, where $Q$ is a $p$-subgroup of $G$ and $\delta\in\Irr(Q)$.
Such a pair $(Q,\delta)$ is called a {\bf Navarro vertex} of $\chi$,
and has been shown to be quite useful in many
problems in the character theory of solvable groups
(see, for example, \cite{C2007,C2008,C2012,CL2012,E2005,La2011,N2002a,N2003}
and the references therein).

Following Navarro, we use $\Irr(G|Q,\delta)$ to denote the set of all members of $\Irr(G)$ with Navarro vertex $(Q,\delta)$, and write  $\IBr_p(G|Q)$ for the set of irreducible $p$-Brauer characters of $G$ with vertex $Q$ (in the sense of Green).
Also, we denote by $G_\delta$ the elements of $N_G(Q)$ that stabilize $\delta$.

One of the important properties of Navarro vertices is that
if we take the vertex character $\delta$ to be the principal character $1_Q$,
then restriction $\chi\mapsto\chi^0$ defines a canonical bijection
$$\Irr(G|Q, 1_Q)\to \IBr_p(G|Q),$$
where $\chi^0$ denotes the restriction of $\chi$ to the set of $p$-regular elements of $G$;
see Theorem A of \cite{N2002}.

Our motivation is to extend the above result
by replacing the principal character $1_Q$ with certain linear characters $\delta$ of $Q$.
We mention that if $\delta$ is linear then $\chi^0$ has vertex $Q$ for each  $\chi\in\Irr(G|Q,\delta)$ with $\chi^0\in\IBr_p(G)$,
and this is the reason why we only consider linear vertex characters.
(Perhaps it is worthy mentioning that Navarro vertices $(Q,\delta)$ with $\delta(1)=1$ also appeared in Theorem B of \cite{WJ2023}.)
For convenience, we say that $\chi\in\Irr(G)$ is a {\bf lift} if $\chi^0\in\IBr_p(G)$,
and that $\delta\in\Irr(Q)$ is {\bf stable} in $G$ if $\delta(x)=\delta(y)$ whenever $x,y\in Q$ are conjugate in $G$. It is easy to see that if $\delta$ is stable in $G$, then $\delta$ is invariant in $N_G(Q)$, so that $G_\delta=N_G(Q)$.

With the notation as above, our main result can be stated as follows.

\begin{thmA}
Let $Q$ be a $p$-subgroup of a $p$-solvable group $G$, where $p$ is a prime,
and suppose that $\delta$ is a linear character of $Q$.
If $\delta$ is stable in $G$,
then restriction to $p$-regular elements defines a canonical bijection $\Irr(G|Q,\delta)\to\IBr_p(G|Q)$.
\end{thmA}

Actually, we will prove Theorem A in the more general setting of Isaacs' $\pi$-partial characters
(see Theorem \ref{main}).
The following Corollary B, which is an important special case of Theorem 2.1 of \cite{N2004},
is useful for applications.
We will establish the $\pi$-version of this result (see Corollary \ref{rdz}),
which can also be used to prove theorems about lifts of $\pi$-partial characters.
For instance, we may replace Theorem 2.3 of \cite{CL2012} by our Corollary \ref{rdz}
to obtain the $\pi$-analogues of Theorems 1.1 and 1.3 of that paper.

\begin{corA}
Let $\varphi\in\IBr_p(G|Q)$, where $G$ is a $p$-solvable group and $Q$ is a $p$-subgroup of $G$. If $Q\NM G$ and $\delta$ is a $G$-invariant linear character of $Q$,
then there is a unique lift $\chi\in\Irr(G)$ of $\varphi$ having Navarro vertex $(Q,\delta)$.
\end{corA}

The next result is another immediate consequence of Theorem A.
In \cite{N2002a}, Navarro proved that
$$|\Irr(G|Q,\delta)|\le |\Irr(G_\delta|Q,\delta)|,$$
and equality holds if either $\delta=1_Q$ or $\delta$ is linear and $Q\in\Syl_p(G)$;
see also \cite{N2003}.
(We mention that it is very hard to prove that equality holds under these two conditions,
because they are essentially the strong forms of the Alperin weight conjecture (see \cite{IN1995}) and the McKay conjecture (see \cite{IN2001}) only for $p$-solvable groups.)
Here, we still have equality whenever $\delta$ is a $G$-stable linear character of $Q$;
see Corollary \ref{=} for a more general $\pi$-version.

\begin{corA}
Let $G$ be a $p$-solvable group for a prime $p$,
$Q$ a $p$-subgroup of $G$, and $\delta$ a linear character of $Q$.
If $\delta$ is stable in $G$, then $|\Irr(G|Q,\delta)|=|\Irr(N_G(Q)|Q,\delta)|$.
\end{corA}

Now we consider lifts of Brauer characters in a $p$-solvable group $G$.
For $N\NM G$, it is known that if $\chi\in\Irr(G)$ is a lift,
then the irreducible constituents of $\chi_N$ can fail,
so some additional conditions are necessary
(see \cite{C2006,CL2011,C2012,CL2012, L2010,N2011}, for example).
The following is an easy application of Theorem A, which should be compared with Corollary 1.3 of \cite{C2012} and Theorem 1.1 of \cite{CL2012}, and we also prove its $\pi$-version as Theorem \ref{app}.

\begin{thmA}
Let $G$ be a $p$-solvable group with $p>2$, let $N\NM G$ and let $\chi\in\Irr(G)$ be a lift with Navarro vertex $(Q,\delta)$.
If $\delta_{Q\cap N}$ is stable in $N$, then every irreducible constituent $\psi$ of $\chi_N$
is also a lift. Moreover, if $\lambda$ is stable in $G$, then $G_\psi=G_{\psi^0}$.
\end{thmA}

Most of the notation and terminology used in this paper are standard,
which can be found in \cite{I1976,N1998,I2018} for ordinary characters,
Brauer characters and $\pi$-partial characters.

\section{Preliminaries}
For the reader's convenience, we will briefly review the definitions of both the Navarro vertices
for ordinary characters and the vertices for $\pi$-partial characters in $\pi$-separable groups;
see \cite{C2010,N2002} for more details.
We begin by introducing some notation.

Given a finite group $G$, we say that $(H,\theta)$ is a {\bf character pair} of $G$
if $H\le G$ and $\theta\in\Irr(H)$, and we define a partial order on the set of character pairs of $G$ by setting $(H,\theta)\le (K,\psi)$ if $H\le K$ and $\theta$ lies under $\psi$.
Observe that $G$ acts on character pairs via conjugation,
that is, $(H,\theta)^g=(H^g,\theta^g)$ for $g\in G$,
where $\theta^g\in\Irr(H^g)$ is defined by $\theta^g(x^g)=\theta(x)$ for $x\in H$.

Suppose that $G$ is $\pi$-separable for a set $\pi$ of primes.
A {\bf $\pi$-factored normal pair} of $G$ is a character pair $(H,\theta)$,
where $H\NM G$ and $\theta$ is $\pi$-factored (see \cite{I2018} for the definition and properties).
We write $\mathcal{N}(G)$ for the set of $\pi$-factored normal pairs of $G$,
and let $\mathcal{N}^*(G)$ be the set of maximal members of $\mathcal{N}(G)$.
We have the following fundamental result (see Sect. 4 of \cite{C2010}),
which is very similar to Theorem 4.1 of \cite{I2018}.

\begin{lem}\label{syl}
Let $\chi\in\Irr(G)$, where $G$ is a $\pi$-separable group. Then the following hold.

{\rm (a)} There exists a unique normal subgroup $N$ of $G$ maximal with the property that
every irreducible constituent $\theta$ of $\chi_N$ is $\pi$-factored.
In particular, $(N,\theta)\in\mathcal{N}^*(G)$.

{\rm (b)} Let $N$ and $\theta$ be as in (a), and assume that $\theta$ is invariant in $G$.
Then $N=G$ and $\theta=\chi$, so $\chi$ is $\pi$-factored.
\end{lem}

Now we can define the {\bf norma nucleus} $(W,\gamma)$ for any $\chi\in\Irr(G)$ recursively, where $G$ is a $\pi$-separable group. If $\chi$ is $\pi$-factored, let $(W,\gamma)=(G,\chi)$.
If $\chi$ is not $\pi$-factored, let $(N,\theta)$ be as in the above Lemma.
Then $\theta$ is not $G$-invariant, so $(G_\theta,\chi_\theta)<(G,\chi)$, where $G_\theta$ is the inertia group of $\theta$ in $G$ and $\chi_\theta$ is the Clifford correspondent of $\chi$ over $\theta$.
We define $(W,\gamma)$ to be any normal nucleus for $\chi_\theta$.
Notice that $(W,\gamma)$ is uniquely determined by $\chi$ up to conjugacy in $G$,
and that $\gamma$ is $\pi$-factored with $\gamma^G=\chi$.
Moreover, let $Q$ be a Hall $\pi'$-subgroup of $W$ and let $\delta$ be the restriction to $Q$ of the $\pi'$-special factor $\gamma_{\pi'}$ of $\gamma$.
Then we say that $(Q,\delta)$ is a {\bf Navarro vertex} for $\chi$.
It is clear that all of the Navarro vertices for $\chi$ are conjugate in $G$
and as before, we write $\Irr(G|Q,\delta)$ for the set of all members of $\Irr(G)$ having Navarro vertex $(Q,\delta)$.

Continuing to assume that $G$ is $\pi$-separable.
Write $G^0$ to denote the set of $\pi$-elements of $G$,
and let $\chi^0$ be the restriction of $\chi$ to $G^0$, where $\chi$ is any character of $G$.
In this situation, we say that $\chi^0$ is a {\bf $\pi$-partial character} of $G$.
A $\pi$-partial character of $G$ is {\bf irreducible} if it cannot be written as a sum of two $\pi$-partial characters, and we write $\I(G)$ for the set of irreducible $\pi$-partial characters of $G$. Also, a character $\chi\in\Irr(G)$ is said to be a {\bf lift} if $\chi^0\in\I(G)$,
and $\chi$ is called a lift of $\varphi$ if $\chi^0=\varphi\in\I(G)$.

Furthermore, we define a {\bf vertex} for $\varphi\in\I(G)$ to be any Hall $\pi'$-subgroup of some subgroup $U\le G$ such that there exists $\mu\in\I(U)$, where $\mu^G=\varphi$ and $\mu(1)$ is a $\pi$-number. We use $\I(G|Q)$ to denote the set of those
irreducible $\pi$-partial characters of $G$ with vertex $Q$.
It is known that the vertex for $\varphi$ is unique up to $G$-conjugacy
(see Theorem 5.17 of \cite{I2018}),
and that if $\pi=p'$, then $\I(G)=\IBr_p(G)$ by the Fong-Swan theorem, so that the vertices for irreducible $\pi$-partial characters are the same as the Green vertices for irreducible Brauer characters.

As expected, the vertices for $\pi$-partial characters behave well with respect to normal subgroups.
\begin{lem}\label{C2009-3.2}
 Suppose that $G$ is $\pi$-separable and $\varphi\in\I(G)$.
 Let $N$ be a normal subgroup of $G$. If $\theta\in\I(N)$ is a constituent of $\varphi_N$,
 then there is a vertex $Q$ for $\varphi$ such that $Q\cap N$ is a vertex for $\theta$.
\end{lem}
\begin{proof}
See Theorem 3.2 of \cite{C2009}.
\end{proof}

We need the following result, which is a $\pi$-analogue of Lemma \ref{syl}.
\begin{lem}\label{max}
Suppose that $G$ is $\pi$-separable, and let $\varphi\in\I(G)$.
Then there exists a unique maximal normal subgroup $N$ of $G$ such that the irreducible constituents of $\varphi_N$ have $\pi$-degree.
Moreover, if $\varphi_N=e\theta$ for some $\theta\in\I(N)$ and integer $e$,
then $N=G$, so that $\varphi$ has $\pi$-degree.
\end{lem}
\begin{proof}
This is the $\pi$-version of Lemma 6.1 of \cite{N2002}, and the proof is parallel.
\end{proof}

The next result, which can fail if the assumption that $2\in\pi$ is dropped,
is key to the proof of Theorem D from the introduction.
\begin{lem}\label{v-n}
Let $G$ be a $\pi$-separable group with $2\in\pi$, and suppose that $\chi\in\Irr(G)$
is a lift with Navarro vertex $(Q,\delta)$.
If $N\NM G$, then some irreducible constituent of $\chi_N$ has Navarro vertex
$(Q\cap N,\delta_{Q\cap N})$.
\end{lem}
\begin{proof}
It follows by the same argument used to prove Theorem 3.1 of \cite{C2012},
together with Theorem 1.2 of \cite{CL2012}.
\end{proof}

At the end of this section, we give an extendibility criterion of characters from Hall subgroups,
which plays a crucial role in our proof of Theorem A.

\begin{lem}\label{ext}
Let $Q$ be a Hall $\pi'$-subgroup of a $\pi$-separable group $G$,
and let $\delta\in\Irr(Q)$.
If $\delta$ is stable in $G$, then there exists a unique $\pi'$-special character $\alpha$ of $G$
such that $\alpha_Q=\delta$.
\end{lem}
\begin{proof}
See Theorem 3.6 and Corollary 3.15 of \cite{I2018}.
\end{proof}

\section{Main results}
We begin with a result concerning stable characters, which is fundamental for our proofs.

\begin{lem}\label{inv}
Let $\gamma\in\Irr(W)$ be $\pi$-factored, where $W$ is a subgroup of a $\pi$-separable group $G$, and suppose that $\theta\in\Irr(N)$ lies under $\gamma$, where $N\NM G$ with $N\le W$.
Let $Q$ be a Hall $\pi'$-subgroup of $W$, and write $\delta=(\gamma_{\pi'})_Q$.
If $\delta$ is a $G$-stable linear character, then $\theta$ is $\pi$-factored, and its $\pi'$-special factor $\theta_{\pi'}$ is a $G$-invariant linear character.
\end{lem}
\begin{proof}
By the basic properties of $\pi$-characters, we see that $\theta$ is $\pi$-factored, and thus $\theta_{\pi'}$ lies under $\gamma_{\pi'}$.
Let $P=Q\cap N$, so that $P$ is a Hall $\pi'$-subgroup of $N$.
Write $\beta=\theta_{\pi'}$. Note that $\gamma_{\pi'}(1)=\delta(1)=1$,
so $(\gamma_{\pi'})_N=\beta$, and thus $\delta_P=\beta_P$.
In particular, $\beta$ is linear.
To show that $\beta$ is invariant in $G$, it suffices to prove that $\beta(x)=\beta(y)$,
where $x,y\in N$ are conjugate in $G$.
Since $\beta$ has $\pi'$-order in the group of linear characters of $N$,
we have $\beta(x)=\beta(x_{\pi'})$ and $\beta(y)=\beta(y_{\pi'})$.
Clearly the $\pi'$-part $x_{\pi'}$ of $x$ is also conjugate to the $\pi'$-part $y_{\pi'}$ of $y$,
so it is no loss to assume that $x,y$ are $\pi'$-elements of $N$.
In this case, we see that each of $x$ and $y$ is conjugate to some element of $P$,
and thus we may assume further that $x,y\in P$. Now $\beta(x)=\delta(x)=\delta(y)=\beta(y)$, as required.
\end{proof}

We need a sufficient condition for a $\pi$-factored normal pair to be maximal.

\begin{lem}\label{max-irr}
Let $\varphi\in\I(G)$, and suppose that $N$ is the unique maximal normal subgroup of $G$ such that the irreducible constituents of $\varphi_N$ have $\pi$-degree.
Assume that $\mu\in\I(N)$ is an irreducible constituent of $\varphi_N$
and $\alpha\in\Irr(N)$ is the $\pi$-special lift of $\mu$.
Then $(N,\alpha\beta)\in\mathcal{N}^*(G)$ for every $\pi'$-special character $\beta\in\Irr(N)$.
\end{lem}
\begin{proof}
We first show that $(N,\alpha)\in \mathcal{N}^*(G)$.
Otherwise, we can choose some $(M,\psi)\in \mathcal{N}(G)$ with $N<M$.
Since any normal constituent of $\pi$-factored characters is $\pi$-factored,
it is no loss to assume that $M/N$ is a chief factor of $G$,
so that $M/N$ is either a $\pi$-group or a $\pi'$-group.
By Corollary 4.5 of \cite{I2018}, we see that every member of $\Irr(M|\alpha)$ is $\pi$-factored.

We will use $\B$-characters to derive a contradiction (see \cite{I2018} for the definition and properties). Let $\chi\in\B(G)$ be a lift of $\varphi$. Then $\chi$ lies over $\alpha$ as $\chi^0=\varphi$ lies over $\alpha^0=\mu$,
and hence there exists $\rho\in\Irr(M)$ satisfying $(N,\alpha)\le (M,\rho)\le (G,\chi)$.
Then $\rho\in\B(M)$, which implies that $\rho$ is $\pi$-special.
It follows that $\rho^0$ is irreducible with $\pi$-degree,
so $\rho^0$ lies under $\varphi$.
Now the maximality of $N$ forces $M\le N$, and this contradiction proves that $(N,\alpha)\in \mathcal{N}^*(G)$.

Now, let $(M,\sigma)\in \mathcal{N}^*(G)$ be such that $(N,\alpha\beta)\le (M,\sigma)$.
Then $(N,\alpha)\le (M,\sigma_\pi)\in\mathcal{N}(G)$, and
by the maximality of $(N,\alpha)$, we have $N=M$. The result follows.
\end{proof}

The following result, in some sense, can be viewed as the converse of the above lemma.

\begin{lem}\label{max-par}
Let $(N,\theta)\in\mathcal{N}^*(G)$, where $\theta_{\pi'}$ is invariant in $G$,
and let $\varphi\in\I(G)$. Suppose that $\chi\in\Irr(G|\theta)$ is a lift of $\varphi$,
and that $\chi$ has a Navarro vertex $(Q,\delta)$ with $\delta(1)=1$.
Then $N$ is the unique maximal normal subgroup of $G$ such that the irreducible constituents of $\varphi_N$ have $\pi$-degree.
\end{lem}
\begin{proof}
By the construction of Navarro vertices discussed in Section 2,
we see that $\theta$ has $\pi$-degree,
and thus $\theta^0=(\theta_\pi)^0$, which is clearly irreducible and lies under $\varphi$.
Suppose that the lemma is false. Then there is a chief factor $M/N$ of $G$ such that
the irreducible constituents of $\varphi_M$ have $\pi$-degree.
Let $\eta\in\I(M)$ lie under $\varphi$ and over $\theta^0$.
Since $\eta(1)$ is a $\pi$-number, there exists a $\pi$-special character $\hat\eta$ of $M$ such that $\hat\eta^0=\eta $ (see Theorem 3.14 of \cite{I2018}).
It is clear that $\hat\eta$ lies over $\theta_{\pi}$, and since $\theta_{\pi'}$ is invariant in $G$
and $M/N$ is either a $\pi$-group or a $\pi'$-group, it follows that there exists a $\pi'$-special character $\beta$ of $M$ lying over $\theta_{\pi'}$ (see Theorem 2.4 of \cite{I2018}).
Now $\hat\eta\beta\in\Irr(M)$ is $\pi$-factored, and it lies over $\theta_\pi\theta_{\pi'}=\theta$, contradicting the maximality of $(N,\theta)$. The proof is complete.
\end{proof}

We are now ready to prove the following result, which is the $\pi$-version of Theorem A.

\begin{thm}\label{main}
Let $G$ be a $\pi$-separable group, $Q$ a $\pi'$-subgroup of $G$,
and $\delta$ a linear character of $Q$. If $\delta$ is stable in $G$,
then restriction $\chi\mapsto\chi^0$ defines a bijection $\Irr(G|Q,\delta)\to\I(G|Q)$.
\end{thm}
\begin{proof}
We will carry out the following steps to complete the proof.

\emph{Step 1. \  For any $\chi\in\Irr(G|Q,\delta)$, we have $\chi^0\in\I(G|Q)$,
so the restriction map is well defined.}

Observe that once we establish that $\chi^0$ is irreducible,
then $Q$ is clearly a vertex of $\chi^0$ as  $\delta$ is linear.
So it suffices to show that $\chi^0\in\I(G)$, and we do this by induction on $|G|$.

By definition, there exists a normal nucleus $(W,\gamma)$ of $\chi$ such that
$Q$ is a Hall $\pi'$-subgroup of $W$ and $(\gamma_{\pi'})_Q=\delta$.
Also, there is some $(N,\theta)\in\mathcal{N}^*(G)$ satisfying
$$(N,\theta)\le (W,\gamma)\le (T,\psi)\le (G,\chi),$$
where $T=G_\theta$ and $\psi\in\Irr(T)$ is the Clifford correspondent of $\chi$ with respect to $\theta$. Then $(W,\gamma)$ is also a normal nucleus of $\psi$, so $\psi\in\Irr(T|Q,\delta)$.
If $T=G$, then $(N,\theta)=(G,\chi)$ by Lemma \ref{syl}, and thus $\chi$ is $\pi$-factored.
Moreover, note that $\gamma(1)=\gamma_\pi(1)$ is a $\pi$-number, so $\chi$ has $\pi$-degree. It follows that $\chi^0$ is irreducible, and we are done in this case.

We can therefore assume that $T<G$. By induction, we have $\psi^0\in\I(T)$.
By Lemma \ref{inv},
we see that $\theta_{\pi'}$ is a $G$-invariant linear character. This implies that $T=G_\theta=G_{\theta_\pi}=G_{\theta^0}$. Notice that $\psi^0$ lies over $\theta^0$,
so the Clifford correspondence for $\pi$-partial characters tells us that $(\psi^0)^G$ is irreducible. Now we have $\chi^0=(\psi^G)^0=(\psi^0)^G$,
and hence $\chi^0\in\I(G)$, as required.

\emph{Step 2. \  The map $\chi\mapsto\chi^0$ is injective.}

With the notation as above, we fix $\tilde\chi\in\Irr(G|Q,\delta)$ with $\tilde\chi^0=\chi^0=\varphi\in\I(G)$, and we want to show that $\tilde\chi=\chi$.
As in Step 1, we can find the following character pairs
$$(\tilde N,\tilde\theta)\le (\tilde W,\tilde\gamma)\le (\tilde T,\tilde\psi)\le (G,\tilde\chi),$$
where $(\tilde N,\tilde\theta)\in\mathcal{N}^*(G)$,
$(\tilde W,\tilde\gamma)$ is a normal nucleus of $\tilde\chi$,
$\tilde T$ is the inertia group of $\tilde\theta$ in $G$, $\tilde\psi$ is the Clifford correspondent of $\tilde\chi$ over $\tilde\theta$,
$Q$ is a Hall $\pi'$-subgroup of $\tilde W$ and $(\tilde\gamma_{\pi'})_Q=\delta$.
By Lemma \ref{inv} again, we see that $\tilde\theta_{\pi'}$ is invariant in $G$,
and by Lemma \ref{max-par}, we have $\tilde N=N$.
Furthermore, both $(\tilde\theta_\pi)^0$ and $(\theta_\pi)^0$ are irreducible constituents of $\varphi_N$, so they are conjugate in $G$. Replacing $\tilde\theta$ by a suitable conjugate,
we can assume that $(\tilde\theta_\pi)^0=(\theta_\pi)^0$, which forces
$\tilde\theta_\pi=\theta_\pi$.
Also, it is easy to see that $(\tilde\theta_{\pi'})_{Q\cap N}=\delta_{Q\cap N}=(\theta_{\pi'})_{Q\cap N}$ (see the proof of Lemma \ref{inv}),
which yields $\tilde\theta_{\pi'}=\theta_{\pi'}$ by Lemma \ref{ext},
and thus $\tilde\theta=\theta$. In particular, we have $\tilde T=T$.
In this case, we obtain $\tilde\psi,\psi\in\Irr(T|Q,\delta)$.
Note that both $\tilde\psi^0$ and $\psi^0$ induce $\varphi$,
so they are irreducible and thus lie over $\theta^0$.
Then the Clifford correspondence for $\pi$-partial characters deduce that
$\tilde\psi^0=\psi^0$. If $T<G$, then $\tilde\psi=\psi$ by induction on $|G|$,
and hence $\tilde\chi=\tilde\psi^G=\psi^G=\chi$, as wanted.
So we assume that $T=G$. Then $N=G$ by Lemma \ref{syl},
which implies that both $\tilde\chi$ and $\chi$ are $\pi$-factored with $\pi$-degree,
and that $Q$ is a Hall $\pi'$-subgroup of $G$. Furthermore, we have $(\tilde\chi_\pi)^0=\varphi=(\chi_\pi)^0$
and $(\tilde\chi_{\pi'})_Q=\delta=(\chi_{\pi'})_Q$.
It follows that
$\tilde\chi_\pi=\chi_\pi$ and $\tilde\chi_{\pi'}=\chi_{\pi'}$ (by Lemma \ref{ext}),
and thus we have $\tilde\chi=\chi$, as required.

\emph{Step 3. \  The map $\chi\mapsto\chi^0$ is surjective.}

Let $\varphi\in\I(G|Q)$. By Lemma \ref{max},
there exists a unique maximal normal subgroup $N$ such that the irreducible constituents of $\varphi_N$ have $\pi$-degree.
We can choose $\mu\in\I(N)$ lying under $\varphi$, such that $\eta\in\I(T|Q)$, where $T=G_\mu$ and $\eta$ is the Clifford correspondent of $\varphi$ over $\mu$
(see Lemma 6.33 of \cite{I2018} or Lemma 6.33 of \cite{I2018}).
In particular, we have $\eta^G=\varphi$.

Suppose first that $T<G$. By induction on $|G|$, there exists $\psi\in\Irr(T|Q,\delta)$ with $\psi^0=\eta$. As before, we have the following character pairs:
$$(M,\sigma)\le (W,\gamma)\le (T,\psi),$$
where $(M,\sigma)\in\mathcal{N}^*(T)$, $(W,\gamma)$ is a normal nucleus for $\psi$,
$Q$ is a Hall $\pi'$-subgroup of $W$ and $(\gamma_{\pi'})_Q=\delta$.
Then $\sigma_{\pi'}$ is invariant in $T$ by Lemma \ref{inv},
and applying Lemma \ref{max-par}, we conclude that $M$ is the unique maximal normal subgroup of $T$ such that the irreducible constituents of $\eta_M$ have $\pi$-degree.
Observe that $\eta_N=e\mu$ for some integer $e$ and $\mu(1)$ is a $\pi$-number,
so $N\le M$.

Furthermore, let $P=Q\cap N$ and $\lambda=\delta_P$.
Then $P$ is a Hall $\pi'$-subgroup of $N$ and
$\lambda$ is also stable in $G$. By Lemma \ref{ext},
there exists a unique $\pi'$-special character $\hat\lambda$ of $N$ extending $\lambda$.
Let $\hat\mu$ be the $\pi$-special lift of $\mu$, and let $\theta=\hat\mu\hat\lambda$.
Then $\theta\in\Irr(N)$ is $\pi$-factored with $\theta_\pi=\hat\mu$ and $\theta_{\pi'}=\hat\lambda$.

We claim that $\theta$ lies under $\gamma$.
To see this, note that $(\gamma_\pi)^0=\gamma^0$ lies under $\psi^0=\eta$
and that $\eta_N=e\mu$, so $\mu$ lies under $(\gamma_\pi)^0$.
This guarantees that $\hat\mu$ must lie under $\gamma_\pi$.
Also, we have
$$((\gamma_{\pi'})_N)_P=(\gamma_{\pi'})_P=\delta_P=
\lambda=\hat\lambda_P.$$
It follows that  $(\gamma_{\pi'})_N=\hat\lambda$ by Lemma \ref{ext},
and hence $\theta=\hat\mu\hat\lambda$ lies under $\gamma=\gamma_\pi\gamma_{\pi'}$,
as claimed.

Now Lemma \ref{max-irr} guarantees that $(N,\theta)\in\mathcal{N}^*(G)$.
Since $\lambda$ is stable in $G$, it follows from Lemma \ref{inv} that $\hat\lambda$ is invariant in $G$. So $G_\theta=G_{\hat\mu}=G_\mu=T$.
Observe that $\psi\in\Irr(T)$ lies over $\gamma$ and hence over $\theta$,
so $\psi^G\in\Irr(G)$ by the Clifford correspondence.
By definition, the nucleus $(W,\gamma)$ of $\psi$ is also a nucleus of $\psi^G$,
so $\psi^G\in\Irr(G|Q,\delta)$. Notice that $(\psi^G)^0=(\psi^0)^G=\eta^G=\varphi$,
and the result follows in this case.

Finally, we consider the remaining case $T=G$. Then $\varphi$ has $\pi$-degree by Lemma \ref{max},
and $Q$ is a Hall $\pi'$-subgroup of $G$.
Let $\alpha$ be the $\pi$-special lift of $\varphi$, and let $\beta$
be the unique $\pi'$-special extension of $\delta$ to $G$ (see Lemma \ref{ext}).
Writing $\chi=\alpha\beta$, we have $\chi\in\Irr(G|Q,\delta)$ and $\chi^0=\alpha^0=\varphi$.
The result follows in this case too.
\end{proof}

As an immediate consequence of the above theorem, we have the following, which is the $\pi$-version of Corollary B in the introduction.

\begin{cor}\label{rdz}
Let $\varphi\in\I(G|Q)$, where $G$ is a $\pi$-separable group.
If $Q\NM G$ and $\delta$ is a $G$-invariant linear character of $G$,
then there is a unique lift $\chi\in\Irr(G)$ of $\varphi$
having $(Q,\delta)$ as a Navarro vertex.
\end{cor}

The next is the $\pi$-analogue of Corollary C from the introduction.

\begin{cor}\label{=}
Let $G$ be a $\pi$-separable group having a nilpotent Hall $\pi'$-subgroup,
and assume that $Q$ is a $\pi'$-subgroup of $G$.
If $\delta$ is a $G$-stable linear character of $Q$, then $|\Irr(G|Q,\delta)|=|\Irr(N_G(Q)|Q,\delta)|$.
\end{cor}
\begin{proof}
By Theorem 6.29 of \cite{I2018} (or Theorem 6.2 of \cite{IN1995}),
we have $|\I(G|Q)|=|\I(N_G(Q)|Q)|$
and the result follows by Theorem \ref{main}.
\end{proof}

Finally, we prove the $\pi$-version of Theorem D with a somewhat stronger form.

\begin{thm}\label{app}
Let $G$ be a $\pi$-separable group with $2\in\pi$, and suppose that $N\NM G$.
Let $\chi\in\Irr(G)$ be a lift with Navarro vertex $(Q,\delta)$, and write $P=Q\cap N$ and $\lambda=\delta_P$. Assume that $\lambda$ is stable in $N$.
Then every irreducible constituent $\psi$ of $\chi_N$ is a lift,
and if in addition $\lambda$ is invariant in $N_G(P)$, then $G_\psi=G_{\psi^0}$.
\end{thm}

\begin{proof}
By Lemma \ref{v-n}, there exists an irreducible constituent $\psi$ of $\chi_N$ having  $(P,\lambda)$ as a Navarro vertex,
and then Theorem \ref{main} implies that $\psi^0\in\I(N)$.

Assume further that $\lambda$ is invariant in $N_G(P)$.
Writing $\eta=\psi^0$, we have $G_\psi\le G_\eta$.
To obtain the reverse containment, let $g\in G_\eta$.
Then both $P^g$ and $P$ are clearly vertices for $\eta$, so they are conjugate in $N$.
It follows that $g\in N_G(P)N$, and hence $g=xn$ for some $x\in N_G(P)$ and $n\in N$.
Observe that $x=gn^{-1}\in G_\eta$ and $N\le G_\psi$.
To prove that $g\in G_\psi$, therefore, it suffices to show that $x\in G_\psi$.
Now $\lambda$ is $N_G(P)$-invariant, so $\psi,\psi^x\in\Irr(N|P,\lambda)$.
By Theorem \ref{main} again, we obtain $\psi=\psi^x$, so $x\in G_\psi$, as desired.
\end{proof}

\section*{Acknowledgements}

This work was supported by the NSF of China (No. 12171289) and the NSF of Shanxi Province (Nos. 20210302123429, 20210302124077).



\end{document}